\documentclass[a4paper]{amsart}
\usepackage[utf8]{inputenc}
\usepackage[english]{babel}

\usepackage{amssymb,amsmath,amsthm,mathtools,esint}
\usepackage[square,numbers]{natbib}
\usepackage{xurl}
\usepackage[hidelinks]{hyperref}
\usepackage[capitalise,nameinlink]{cleveref}
\crefname{equation}{}{}
\crefformat{equation}{\textup{#2(#1)#3}}

\newtheorem{theorem}{Theorem}
\newtheorem{corollary}[theorem]{Corollary}

\title[Local approximation and super-localization]{Local Approximation of $A$-Harmonic Functions
and the Super-Localization Conjecture}
\author[Moritz Hauck]{Moritz Hauck${}^{*}$}
\address{${}^{*}$Institute for Applied and Numerical Mathematics, Karlsruhe Institute of Technology, Englerstr.~2, 76131 Karlsruhe, Germany}
\email{moritz.hauck@kit.edu}

\author[Chupeng Ma]{Chupeng Ma${}^{\dagger}$}
\address{${}^{\dagger}$School of Sciences, Great Bay University, Dongguan 523000, China}
\email{chupeng.ma@gbu.edu.cn}

\author[Daniel Peterseim]{Daniel Peterseim${}^{\ddagger}$}
\address{${}^{\ddagger}$Institute of Mathematics \& Centre for Advanced Analytics and Predictive Sciences (CAAPS), University of Augsburg, Universit\"atsstr.~12a, 86159 Augsburg, Germany}
\email{daniel.peterseim@uni-a.de}


\begin{document}

\begin{abstract}
We consider the local energy approximation of $A$-harmonic functions on
concentric cubes with a fixed size ratio in $\mathbb{R}^d$, $d\ge2$,
for bounded, symmetric, uniformly elliptic diffusion tensors~$A$.
We give a short proof of the Kolmogorov $n$-width decay rate
$\exp(-cn^{1/(d-1)})$ conjectured by Ma
[Found.\ Comput.\ Math.\ 26 (2026), 1699--1758], where $n$ denotes
the approximation dimension. We apply this result to numerical homogenization, improving local approximation estimates for
multiscale spectral generalized finite element methods and proving
the conjectured local super-localization rate of Hauck and Peterseim
[Math.\ Comp.\ 92 (2023), 981--1003] for interior Cartesian patches.
\end{abstract}

\maketitle

\section{Introduction}
\enlargethispage{3pt}

Numerical homogenization aims at the efficient approximation of elliptic multiscale problems on coarse computational scales, without resolving all fine-scale features globally. A recurring local problem is the approximation of $A$-harmonic functions on a subdomain from information on a larger domain. 
In the multiscale spectral generalized finite element method (MS-GFEM), Babu\v{s}ka and Lipton~\cite{MR2801210} formulate this problem through the Kolmogorov $n$-widths of a restriction operator, characterize optimal approximation spaces by a generalized eigenvalue problem, and prove nearly exponential convergence. The AL basis of Grasedyck, Greff, and Sauter~\cite{MR2902606} likewise uses local approximation of $A$-harmonic functions to construct a numerical
homogenization method. The localized orthogonal decomposition constructs operator-adapted basis functions and truncates their exponentially decaying corrections~\cite{MR3246801}. Subspace-decomposition arguments explain this decay through a uniformly contracting local iteration~\cite{MR3536998,MR3834684}. These localization estimates also yield bounds for the Kolmogorov $n$-widths, as shown in~\cite[Section~6.3]{MR4700411}.

Conversely, local approximation estimates can be used to design sources with localized responses. A normalized coarse source that annihilates a good approximation space for $A$-harmonic functions has a small pairing with every normalized $A$-harmonic function. These pairings control the difference between its global elliptic response and its patch-local response. Hauck and Peterseim~\cite{MR4550317} exploit this characterization in the super-localized orthogonal decomposition (SLOD) by optimizing coarse sources and conjecture super-exponential decay with respect to the oversampling parameter (the number of element layers in a patch). The associated approximation question concerns the decay of the Kolmogorov $n$-widths, where $n$ denotes the dimension of the approximation space. Babu\v{s}ka and Lipton~\cite{MR2801210} prove the rate $\exp(-n^{1/(d+1)-\epsilon})$, with $d$ the spatial dimension, for every $\epsilon>0$. Ma~\cite{MR5010382} improves the exponent to $1/d$ within a unified framework for MS-GFEMs and conjectures the exponent $1/(d-1)$. For the Laplacian, Mohr~\cite{Mohr2026} proves the 
local SLOD decay estimate corresponding to this exponent. A proof for general bounded, symmetric, uniformly elliptic diffusion tensors was recently presented in~\cite{Jin2026}.

We give a short alternative proof using the energy-determinant idea
of Li and Wang~\cite{MR1804533}, also used in~\cite{Jin2026}.
Instead of boundary spectral estimates, our argument relies on
local Poincar\'e inequalities and Hadamard's inequality applied to
successive layers. This yields an explicit numerical constant and
improves the dependence on the coefficient contrast in the exponent.
We discuss the consequences for MS-GFEM and the super-localized
generalized finite element method~\cite{MR4700411}, and prove
the conjectured local super-localization rate of Hauck and
Peterseim~\cite{MR4550317}.

\section{Local approximation of $A$-harmonic functions}

Let $d\ge 2$ and $r>0$, and define concentric cubes 
\[ \omega\coloneqq (-r,r)^d, \qquad \omega^*\coloneqq (-3r,3r)^d.\]
This setting suffices, since translation, rescaling, and a finite
covering argument extend the approximation result to compactly nested bounded
Lipschitz domains, with geometry-dependent constants.
Further, let
\[A\in L^\infty(\omega^*;\mathbb{R}_{\mathrm{sym}}^{d\times d})\]
satisfy the uniform ellipticity bounds
\begin{equation}\label{eq:coeffbounds}
    \alpha|\xi|^2 \le (A(x)\xi)\cdot\xi \le \beta|\xi|^2 
\end{equation}
for almost every $x\in\omega^*$ and every $\xi\in\mathbb{R}^d$, where $0<\alpha\le\beta<\infty$.

For every open set $G\subseteq\omega^*$, define
\[a_G(u,v)\coloneqq \int_G (A(x)\nabla u(x))\cdot\nabla v(x)\,dx\]
and the associated energy seminorm
\[\|u\|_{a,G}\coloneqq a_G(u,u)^{1/2}.\]
Furthermore, let
\[\mathcal H(G)\coloneqq \left\{u\in H^1(G):a_G(u,\varphi)=0\quad\text{for all }\varphi\in H^1_0(G)\right\}.
\]

The following theorem confirms the decay rate conjectured by
Ma~\cite{MR5010382} in the present setting.

\begin{theorem}[Kolmogorov $n$-width of restricted $A$-harmonic functions]
\label{thm:main}
For every $n\in\mathbb N$ with $n+1\ge 2(3^d-1)$, the Kolmogorov $n$-width
\[
d_n\coloneqq \adjustlimits\inf_{\substack{V_n\subset\mathcal H(\omega)\\
                \dim V_n= n}}
\sup_{\substack{u\in\mathcal H(\omega^*)\\
                \|u\|_{a,\omega^*}=1}}
\;
\mathop{\inf\vphantom{\sup}}_{v\in V_n}
\|u-v\|_{a,\omega}
\]
satisfies
\[
d_n
\le
\exp\left(
-c\sqrt{\frac{\alpha}{\beta}}\,
(n+1)^{1/(d-1)}
\right),
\]
where $c\coloneqq \bigl(128(1+\pi^{-1})\bigr)^{-1}$.
\end{theorem}

Note that $d_n$ is at most one for every $n\in\mathbb{N}$, since $v=0$ is an admissible approximant in every space $V_n$ and restriction does not increase the energy.

\Cref{thm:main} concerns interior subdomains. The case where $\omega^*$ intersects the boundary of a domain~$\Omega$ is similar. We replace $\omega$ and $\omega^*$ by $\omega\cap\Omega$ and $\omega^*\cap\Omega$ and impose homogeneous Dirichlet conditions on $\partial\Omega\cap\omega^*$ in the definition of~$\mathcal H$. The proof then carries over with minor modifications similar to those in~\cite{MR5010382,MaScheichlDodwell2022}.

\subsection{Proof of \texorpdfstring{\cref{thm:main}}{Theorem 1}}
The quantity $d_n$ is the Kolmogorov $n$-width of the restriction operator
\begin{equation}
	\label{eq:P}
	P:\mathcal H(\omega^*)\longrightarrow\mathcal H(\omega),
	\qquad
	Pu=u|_\omega,
\end{equation}
with respect to the energy seminorms. For the Hilbert-space argument,
functions differing by a constant are identified. This does not
change $d_n$.

The operator $P$ is compact; see~\cite[Appendix~A]{MR2801210}.
The spectral characterization
in~\cite[Theorems~3.1--3.2]{MR2801210} gives
\begin{equation}
d_n=\sqrt{\lambda_{n+1}},
\label{eq:width-eigenvalue}
\end{equation}
where $\lambda_1\ge\lambda_2\ge\cdots\ge0$ are the eigenvalues of $P^*P$, with $P^*$ the adjoint of~$P$ with respect to the energy inner products. The corresponding $a_{\omega^*}$-orthonormal eigenfunctions $u_i\in\mathcal H(\omega^*)$ satisfy
\[
a_\omega(u_i,v)=\lambda_i a_{\omega^*}(u_i,v)
\qquad\text{for every }v\in\mathcal H(\omega^*).
\]
If $P$ has finite rank, the sequence of eigenvalues is extended by zeros.

Fix $n$ such that $n+1\ge2(3^d-1)$, as required by the theorem. If $\lambda_{n+1}=0$, the assertion follows immediately.
Otherwise, fix representatives of the first $n+1$ eigenfunctions and set
\[
V\coloneqq \operatorname{span}\{u_1,\ldots,u_{n+1}\}.
\]
These functions and their representatives remain fixed throughout
the argument.

We choose
\[
K\coloneqq 
\left\lfloor
\left(\frac{n+1}{2(3^d-1)}\right)^{1/(d-1)}
\right\rfloor,
\]
where $\lfloor\cdot\rfloor$ denotes the floor function. Then, $K\ge1$ and
\begin{equation}
    \label{eq:estL}
    (3^d-1)K^{d-1}\le\frac{n+1}{2},
\qquad
K\ge\frac12
\left(\frac{n+1}{2(3^d-1)}\right)^{1/(d-1)}.
\end{equation}

We define
\[
h\coloneqq \frac{2r}{K},
\qquad
\omega_k\coloneqq (-r-kh,r+kh)^d,
\qquad k=0,\ldots,K.
\]
Thus, $\omega_0=\omega$ and $\omega_K=\omega^*$.
Partition $\omega^*$ into open Cartesian cubes of side length $h$,
up to their boundaries, and denote this partition by $\mathcal T_h$.
Each $\omega_k$ is a union of cells of $\mathcal T_h$, up to their boundaries.

For $k=1,\ldots,K$, set
\[
S_k\coloneqq \omega_k\setminus\overline{\omega_{k-1}},
\qquad
\mathcal T_h(S_k)\coloneqq \{T\in\mathcal T_h:T\subset S_k\},
\qquad
m_k\coloneqq \#\mathcal T_h(S_k).
\]
The cube $\omega_k$ contains exactly $(K+2k)^d$ cells.
Hence,
\begin{equation}
m_k=(K+2k)^d-(K+2k-2)^d,
\qquad
\sum_{k=1}^K m_k=(3^d-1)K^d.
\label{eq:exact-layer-count}
\end{equation}

On each layer, define
\[
W_k\coloneqq 
\left\{
v\in V:
\int_T v\,dx=0
\quad\text{for every }T\in\mathcal T_h(S_k)
\right\}.
\]
These are $m_k$ linear conditions, so
\begin{equation}
\dim W_k\ge(n+1-m_k)_+,
\label{eq:layer-dimension}
\end{equation}
where $t_+\coloneqq \max\{t,0\}$.

The Poincaré constant on $(0, 1)^d$ for functions with vanishing mean is $\pi^{-1}$. For every $v \in W_k$, scaling and summing the cellwise Poincaré inequalities yields
\begin{equation}
\|v\|_{L^2(S_k)}
\le \pi^{-1} h\|\nabla v\|_{L^2(S_k)}.
\label{eq:layer-poincare}
\end{equation}
Choose
\[
\eta_k(x)\coloneqq 
\min\left\{
1,\,
\max\left\{0,\frac{r+kh-\|x\|_\infty}{h}\right\}
\right\}.
\]
This cutoff equals one on $\omega_{k-1}$, vanishes outside
$\omega_k$, and satisfies $|\nabla\eta_k|\le h^{-1}$
almost everywhere. In particular,
$\eta_k v\in H^1_0(\omega^*)$ for every $v\in V$.

For $v\in W_k$, weak harmonicity and
\cref{eq:layer-poincare} yield
\begin{align*}
\|v\|_{a,\omega_{k-1}}^2
&\le
\int_{\omega^*}\eta_k (A\nabla v)\cdot\nabla v\,dx=
-\int_{S_k}v(A\nabla v)\cdot\nabla\eta_k\,dx\\
&\le
\frac{\sqrt\beta}{h}
\|v\|_{L^2(S_k)}\|v\|_{a,S_k}\le
\pi^{-1}\sqrt{\frac{\beta}{\alpha}}\,
\|v\|_{a,S_k}^2,
\end{align*}
where we used the coefficient bounds in~\cref{eq:coeffbounds}. Adding the inner and shell energies therefore gives
\begin{equation}
\|v\|_{a,\omega_{k-1}}^2
\le\theta\|v\|_{a,\omega_k}^2,
\qquad
\theta\coloneqq 
\frac{\pi^{-1}\sqrt{\beta/\alpha}}
     {1+\pi^{-1}\sqrt{\beta/\alpha}}
<1.
\label{eq:layer-contraction}
\end{equation}
Following an idea of Li and Wang~\cite{MR1804533}, we compare
the energy forms on the fixed space $V$ through their determinants.
Define
\[
G_k\coloneqq 
\bigl(a_{\omega_k}(u_i,u_j)\bigr)_{i,j=1}^{n+1},
\qquad k=0,\ldots,K.
\]
Their endpoint values are
\begin{equation}
G_0=\operatorname{diag}(\lambda_1,\ldots,\lambda_{n+1}),
\qquad
G_K=I_{n+1}.
\label{eq:endpoint-gramians}
\end{equation}
Since $\lambda_{n+1}>0$ and $G_k-G_0$ is positive semidefinite,
every $G_k$ is positive definite.

Fix a layer $k$. Choose an $a_{\omega_k}$-orthonormal basis
$\{v_1,\ldots,v_{n+1}\}$ of $V$ whose first $\dim W_k$ members span
$W_k$. Let $Q_k$ be the corresponding change-of-basis matrix.
Then
\[
Q_k^\top G_k Q_k=I_{n+1},
\qquad
B_k\coloneqq Q_k^\top G_{k-1}Q_k
=
\bigl(a_{\omega_{k-1}}(v_i,v_j)\bigr)_{i,j=1}^{n+1}.
\]
By \cref{eq:layer-contraction}, the first $\dim W_k$ diagonal
entries of $B_k$ are at most $\theta$. All remaining diagonal
entries are at most one by monotonicity of the energy under
restriction. Using the multiplicativity of the determinant and Hadamard's inequality for positive definite matrices~\cite[Theorem~7.8.1]{HornJohnson2013} therefore gives
\begin{equation}
\frac{\det G_{k-1}}{\det G_k}
=
\det B_k
\le
\prod_{i=1}^{n+1}(B_k)_{ii}
\le
\theta^{\dim W_k}
\le
\theta^{(n+1-m_k)_+}.
\label{eq:layer-determinant}
\end{equation}

Multiplying these estimates over all layers and using
\cref{eq:endpoint-gramians} yields
\[
\prod_{i=1}^{n+1}\lambda_i
=
\frac{\det G_0}{\det G_K}
\le
\theta^{\sum_{k=1}^K(n+1-m_k)_+}.
\]
By the exact count \cref{eq:exact-layer-count},
\[
\sum_{k=1}^K(n+1-m_k)_+
\ge
(n+1)K-\sum_{k=1}^K m_k
=
(n+1)K-(3^d-1)K^d.
\]
Since $0<\theta<1$, we conclude that
\begin{equation}
\prod_{i=1}^{n+1}\lambda_i
\le
\theta^{(n+1)K-(3^d-1)K^d}.
\label{eq:product-bound}
\end{equation}

Using the ordering of the eigenvalues, \cref{eq:product-bound}, and the first estimate in \cref{eq:estL}, we obtain
\[
\lambda_{n+1}^{n+1}
\le
\prod_{i=1}^{n+1}\lambda_i
\le
\theta^{(n+1)K/2}.
\]
Consequently,
\[
d_n=\sqrt{\lambda_{n+1}}\le\theta^{K/4}.
\]
Since $\alpha \le \beta$, the definition of $\theta$ and the elementary inequality $\log(t) \ge \frac{t-1}{t}$ for $t > 0$ yield
\[
-\log\theta
=
\log\left(1+\pi\sqrt{\frac{\alpha}{\beta}}\right)
\ge
\frac{1}{1+\pi^{-1}}\sqrt{\frac{\alpha}{\beta}}.
\]
Therefore, the second estimate in \cref{eq:estL} yields
\[
d_n
\le
\exp\left(
-\frac{1}{8(1+\pi^{-1})}
\sqrt{\frac{\alpha}{\beta}}
\left(\frac{n+1}{2(3^d-1)}\right)^{1/(d-1)}
\right).
\]
This and the elementary estimate $\big(2(3^d-1)\big)^{1/(d-1)} \leq 16$, which holds for $d \geq 2$, readily yield the assertion.

\subsection{Consequences for MS-GFEM}
In the original MS-GFEM of Babu\v{s}ka and Lipton~\cite{MR2801210}, the local approximation spaces on the subdomain~$\omega$ are built from $A$-harmonic functions on the oversampling domain~$\omega^*$. These spaces are optimal in the sense of Kolmogorov $n$-widths. With $n$ local basis functions, the local approximation error equals the $n$-width~$d_n$ from \cref{thm:main}. For concentric cubes, \cref{thm:main} thus bounds this error by $\exp\bigl(-c\sqrt{\alpha/\beta}\,(n+1)^{1/(d-1)}\bigr)$. For compactly nested bounded Lipschitz domains, the covering argument mentioned at the beginning of this section yields the same rate, with constants depending additionally on the geometry. This improves the exponent $1/d$ established in~\cite{MR5010382}.

More recent versions of the method~\cite{MR5010382,MaScheichlDodwell2022} additionally multiply the restricted $A$-harmonic functions by the partition of unity function~$\chi$ associated with~$\omega$. Their local spaces are spanned by the first $n$ singular vectors of the compact operator
\[
P_\chi\colon \mathcal{H}(\omega^{\ast})\rightarrow H_{0}^{1}(\omega),
\qquad
P_\chi u\coloneqq \chi\,(u|_{\omega}),
\]
where, as before, $\mathcal{H}(\omega^{\ast})$ is equipped with the energy norm and functions differing by a constant are identified. The local approximation error is then bounded by the $n$-width $d_n(P_\chi)$ of this operator. It is straightforward to show that $d_n(P_\chi)$ is at most a constant multiple of~$d_n$, with a constant depending only on~$\chi$. Hence, these methods achieve the same decay rate as the original MS-GFEM.


\section{Super-localization}
For the application to the SLOD~\cite{MR4550317}, we restrict ourselves to a Cartesian mesh with mesh size $H>0$, although more general settings could be considered. Each basis function is computed on an oversampling patch consisting of an element and $\ell\in\mathbb N$ surrounding element layers, where $\ell$ is the oversampling parameter. Centering the patch at the origin, we have
\[
\omega^*=\bigl(-\tfrac{2\ell+1}{2}H,\tfrac{2\ell+1}{2}H\bigr)^d.
\]
We denote by $\mathcal T_H(\omega^*)$ the partition of $\omega^*$ into
$(2\ell+1)^d$ open Cartesian cubes of side length $H$, and by
$\mathbb Q^0_H(\omega^*)$ the space of $\mathcal T_H(\omega^*)$-piecewise
constant functions.

The basis function associated with $\omega^*$ is the patch-local
response to a piecewise constant source $g$ attaining the infimum in
\begin{equation}
\label{eq:sigma}
\sigma_{H,\ell}\coloneqq 
\adjustlimits\inf_{\substack{g\in\mathbb Q^0_H(\omega^*)\\
                \|g\|_{L^2(\omega^*)}=1}}
\sup_{\substack{v\in\mathcal H(\omega^*)\\
                \|v\|_{H^1(\omega^*)}=1}}
\;
(g,v)_{L^2(\omega^*)},
\end{equation}
where $\|v\|_{H^1(\omega^*)}^2\coloneqq \|v\|_{L^2(\omega^*)}^2+\|\nabla v\|_{L^2(\omega^*)}^2$.
The quantity $\sigma_{H,\ell}$ controls the resulting localization
error; see~\cite[Theorem~6.1]{MR4550317}. 

\begin{corollary}[Super-exponential localization]
\label{cor:slod}
For every $\ell\ge1$,
\[
\sigma_{H,\ell}
\le
\frac{(\ell+2)H}{\pi}\sqrt{\frac{\beta}{\alpha}}\,
\exp\left(
-\frac c8\sqrt{\frac{\alpha}{\beta}}\,
\ell^{d/(d-1)}
\right),
\]
where $c$ is the constant from \cref{thm:main}.
\end{corollary}

\begin{proof}
The proof strategy is inspired by the proof of~\cite[Theorem~7.3]{MR4550317}. First, let $\ell\ge7$. We apply \cref{thm:main} to $\omega^*$ and
the concentric cube
\[
\omega=\bigl(-\tfrac{2\ell+1}{6}H,\tfrac{2\ell+1}{6}H\bigr)^d.
\]
A cell of $\mathcal T_H(\omega^*)$ with center $Hz$, $z\in\mathbb Z^d$,
lies in $\omega$ if and only if $\|z\|_\infty+\frac12\le\frac{2\ell+1}{6}$,
that is, $\|z\|_\infty\le(\ell-1)/3$. Hence, the submesh
$\mathcal T_H(\omega)\coloneqq \{T\in\mathcal T_H(\omega^*):T\subset\omega\}$
consists of $(2m+1)^d$ cells, where $m\coloneqq \lfloor(\ell-1)/3\rfloor$. Let $\mathbb Q^0_H(\omega)\subset\mathbb Q^0_H(\omega^*)$ denote the subspace of functions that vanish on all cells not in~$\mathcal T_H(\omega)$. We set
$n\coloneqq (2m+1)^d-2$, so that $\dim\mathbb Q^0_H(\omega)=n+2$. Since
$\ell\ge7$, we have $m\ge2$ and thus $n+1\ge5^d-1\ge2(3^d-1)$, as
required in \cref{thm:main}.

Let $V_n\subset\mathcal H(\omega)$ with $\dim V_n = n$ be an optimal
space in the definition of~$d_n$, which exists by~\cite[Theorems~3.1--3.2]{MR2801210}. For the $n+2$ cell values of
$g\in\mathbb Q^0_H(\omega)$, $L^2(\omega)$-orthogonality to
$V_n+\operatorname{span}\{1\}$ amounts to at most $n+1$ linearly independent conditions. Hence, there exists a
solution $g$ with $\|g\|_{L^2(\omega^*)}=1$. It need not be unique, and
any such $g$ will do.

Consider $v\in\mathcal H(\omega^*)$ and $q\in V_n$. Since $g$ vanishes
outside $\omega$ and is orthogonal to $q$ and to the constants, the
Cauchy--Schwarz inequality, the Poincaré inequality on the cube $\omega$
of side length $(2\ell+1)H/3$, and~\cref{eq:coeffbounds} yield
\begin{align*}
(g,v)_{L^2(\omega^*)}
&=
\Bigl(g,\,v-q-\fint_\omega(v-q)\,dx\Bigr)_{L^2(\omega)}
\le
\Bigl\|v-q-\fint_\omega(v-q)\,dx\Bigr\|_{L^2(\omega)}\\
&\le
\frac{(2\ell+1)H}{3\pi}\|\nabla(v-q)\|_{L^2(\omega)}
\le
\frac{(\ell+2)H}{\pi\sqrt\alpha}\|v-q\|_{a,\omega},
\end{align*}
using that $2\ell+1\le3(\ell+2)$.
Taking the infimum over $q\in V_n$, the optimality of $V_n$ and~\cref{eq:coeffbounds} yield
\[
(g,v)_{L^2(\omega^*)}
\le
\frac{(\ell+2)H}{\pi\sqrt\alpha}\,d_n\|v\|_{a,\omega^*}
\le
\frac{(\ell+2)H}{\pi}\sqrt{\frac{\beta}{\alpha}}\,d_n\|v\|_{H^1(\omega^*)}.
\]
Since $g$ is admissible in the definition of $\sigma_{H,\ell}$, taking the
supremum over all $v\in\mathcal H(\omega^*)$ with
$\|v\|_{H^1(\omega^*)}=1$ and applying \cref{thm:main} yields
\[
\sigma_{H,\ell}
\le
\frac{(\ell+2)H}{\pi}\sqrt{\frac{\beta}{\alpha}}\,d_n
\le
\frac{(\ell+2)H}{\pi}\sqrt{\frac{\beta}{\alpha}}\,
\exp\left(
-c\sqrt{\frac{\alpha}{\beta}}\,
(n+1)^{1/(d-1)}
\right).
\]

Finally, using that $m=\lfloor(\ell-1)/3\rfloor\ge(\ell-3)/3$, we obtain $2m+1\ge(2\ell-3)/3\ge\ell/2$ for $\ell\ge6$. Together with $(2m+1)^d\ge2$ and $(d+1)/(d-1)\le3$, this gives
\begin{align*}
(n+1)^{1/(d-1)}
&=
\bigl((2m+1)^d-1\bigr)^{1/(d-1)}
\ge
\bigl(\tfrac12(2m+1)^d\bigr)^{1/(d-1)}
\ge
\bigl(\tfrac12(\ell/2)^d\bigr)^{1/(d-1)}\\
&=
2^{-(d+1)/(d-1)}\,\ell^{d/(d-1)}
\ge
\tfrac18\,\ell^{d/(d-1)}.
\end{align*}
Inserting this into the previous bound yields the assertion for $\ell\ge7$.

Now, let $1\le\ell\le6$. We choose two cells of $\mathcal T_H(\omega^*)$ that share a face. Together, they form a box~$B$ with side lengths $2H,H,\ldots,H$. Let $g\in\mathbb Q^0_H(\omega^*)$ equal a positive constant on one of the two cells, the negative of this constant on the other, and zero elsewhere, where the constant is chosen such that $\|g\|_{L^2(\omega^*)}=1$. Then $g$ vanishes outside $B$ and has vanishing mean. By the definition of $\sigma_{H,\ell}$ in~\cref{eq:sigma}, the
Cauchy--Schwarz inequality, the Poincaré inequality on $B$ with constant $2H/\pi$, and $\|\nabla v\|_{L^2(\omega^*)}\le\|v\|_{H^1(\omega^*)}$ yield
$\sigma_{H,\ell}\le2H/\pi$. Moreover,
$\sqrt{\alpha/\beta}\,\ell^{d/(d-1)}\le\ell^2\le36$ and
$9c/2\le1/32$. Hence, the elementary inequality $\exp(-t)\ge1-t$ for $t\in\mathbb R$ shows that the exponential factor in the assertion is at least
\[
\exp(-9c/2)\ge1-\tfrac{9c}{2}\ge\tfrac{31}{32}\ge\tfrac23\ge\tfrac{2}{\ell+2},
\]
which, together with $\sqrt{\beta/\alpha}\ge1$, yields the assertion.
\end{proof}

The localization contribution to the error of the super-localized
generalized finite element method is likewise controlled by
$\sigma_{H,\ell}$; see~\cite[Theorem~5.3]{MR4700411}.
Under the assumptions of that theorem, \cref{cor:slod} yields
super-exponential decay of this contribution with respect
to the oversampling parameter~$\ell$.
\bigskip

\paragraph{\textbf{Acknowledgments.}}
This work grew out of discussions during the 9th Chinese--German Workshop on Computational and Applied Mathematics in Bonn (October 5--9, 2026), where the authors became aware of their independent progress and combined their approaches. They thank the organizers for the stimulating environment. M. Hauck acknowledges funding from the Deutsche Forschungsgemeinschaft (DFG, German Research Foundation) -- Project-ID 258734477 -- SFB~1173.

\medskip
\paragraph{\textbf{Use of large language models.}}
ChatGPT was used to develop the central arguments in the proof of
\cref{thm:main}. In particular, it identified the
energy-determinant argument of Li and Wang~\cite{MR1804533}
and suggested adapting it to the local approximation problem.
Claude was used for language
editing and to check and refine the constants in the proof of \cref{cor:slod}.
The authors have checked all mathematical arguments and take
full responsibility for the content of the paper.

\bibliographystyle{abbrv}
\bibliography{refs}
\end{document}